\documentclass[11pt, reqno]{amsart}       
\usepackage{graphicx}

\usepackage{amsmath}

\usepackage{amssymb}
\usepackage{pifont}
\newtheorem{thm}{Theorem}
\newtheorem{lem}{Lemma}

\newtheorem{cor}{Corollary}
\theoremstyle{definition}

\newtheorem*{thm*}{Theorem}

\newtheorem*{rem2}{Remark}
\newtheorem*{example}{Example}
\newtheorem*{acknowledgement}{Acknowledgement}

\newcommand{\vertiii}[1]{{\left\vert\kern-0.25ex\left\vert\kern-0.25ex\left\vert
		#1 \right\vert\kern-0.25ex\right\vert\kern-0.25ex\right\vert}}

\def \min   {\text {\rm min}}

\def \lim   {\text {\rm lim}}

\begin{document}
	
	\title[]{Some non-commutative averaging theorems}
	
	\author{Saptak Bhattacharya}
	
	\address{Indian Statistical Institute\\
		New Delhi 110016\\
		India}
	\email{saptak21r@isid.ac.in}
	
	
	
	
	\begin{abstract}
		Given $n\in\mathbb{N}$ any point on the closed unit disk $\overline{\mathbb{D}}$ can be written as the average of $n$ points on the unit circle $\mathbb{S}^1$. Here we discuss a non-commutative version of this result. We prove that for any Hilbert space $\mathcal{H}$ and a state $\phi:B(\mathcal{H})\to\mathbb{C}$, $\{\phi(U): U\,\mathrm{ unitary}\}=\overline{\mathbb{D}}$. We also show that if $\mathrm{dim }\,\mathcal{H}\,\textless\,\infty$, for any $w\in\overline{\mathbb{D}}$ we can choose a unitary $U$ with atmost $3$ distinct eigenvalues such that $\phi(U)=w$. Lastly, we prove the {\it divisibility property} for any state $\phi$ on $B(\mathcal{H})$ where $\mathcal{H}$ is infinite-dimensional, showing that $\{\phi(P) : P^*=P^2=P\}=[0,1]$. 
	\end{abstract}
	
	\date{}
	\maketitle
	
	\section{introduction}
	
	We start off with the following question : what points $w$ on the closed unit disk $\overline{\mathbb{D}}=\{z\in\mathbb{C}:|z|\leq 1\}$ can be written as the average of $n$ points on the unit circle? To be precise, consider the map $\xi$ on the $n$-torus $\mathbb{T}^n$ given by $$\xi(z_1,\dots, z_n)=\frac{\sum_{j=1}^n z_j}{n}$$ where $|z_j|=1$ for all $j$. The problem now asks to find the range of $\xi$. To resolve this, note that $\mathrm{im}\,\xi$ satisfies two main properties :\\
	
	\textbf{1.} If $z\in\mathrm{im}\, \xi$, $e^{it} z\in\mathrm{im}\, \xi$ for all $t$. Thus, $\mathrm{im}\,\xi$ is {\it rotationally invariant}.\\
	
	\textbf{2.} $\xi(1,\omega,\dots,\omega^{n-1})=0$ and $\xi(1,1,\dots,1)=1$, where $\omega\neq 1$ is a complex $n^{th}$ root of unity.
	\medskip
	
	Hence, by the intermediate value property, $[0,1]\subset\mathrm{im}\,|\xi|$. Let $r\in[0,1]$ and choose $x\in\mathbb{T}^n$ such that $|\xi(x)|=r$. Write $\xi(x)=e^{it_0}r$ to show that $\xi(e^{-it_0}x)=r$. Thus, $[0,1]\in\mathrm{im}\,\xi$. By rotational invariance, \[\mathrm{im}\,\xi=\overline{\mathbb{D}}=\{z\in\mathbb{C}:|z|\leq 1\}.\label{e1}\tag{1}\]
	\medskip
	
	Let us view this problem through a different lens now. Consider the $C^*$-algebra $l^{\infty}_n$ of $n$-tuples of complex numbers $(z_1,\dots, z_n)$ with the norm $$||(z_1,\dots, z_n)||=\max_{j}|z_j|.$$ The set of all unitaries in this algebra is precisely $\mathbb{T}^n$. Consider the state $\phi:l^{\infty}_n\to\mathbb{C}$ given by $$\phi(z_1,\dots,z_n)=\frac{\sum z_j}{n}.$$ Then the range of this state, when restricted to the unitaries, is precisely $\overline{\mathbb{D}}$. However, this is not true for an arbitrary state. For example, let $\lambda\in(0,1/2)$ and consider the state $\phi:l^{\infty}_2\to\mathbb{C}$ given by $$\phi(z,w)=\lambda z+(1-\lambda)w.$$ When restricted to $\mathbb{T}^2$, the range of this state does not contain $0$. Note that this is the only thing that prevents the range from being the whole of $\overline{\mathbb{D}}$ because rotational invariance is still there and $\phi(1,1)=1$.
	\medskip
	
	Our main results discuss these problems in the non-commutative setup. Much unlike the commutative case, any state $\phi:M_n(\mathbb{C})\to\mathbb{C}$, when restricted only to the unitaries with atmost $3$ eigenvalues, cover the whole of $\overline{\mathbb{D}}$. If a Hilbert space $\mathcal{H}$ is infinite dimensional, then again, a state $\phi$ on $B(\mathcal{H})$ restricted to the unitaries yield $\overline{\mathbb{D}}$. Not only this, if $\phi$ is restricted to the lattice of projections in $B(\mathcal{H})$, we get the entire closed interval $[0,1]$. This is significant, since infinite dimensional states can be pathological, especially when they are not normal. As a consequence, we deduce that if $P\in B(\mathcal{H})$ is a projection, $$\{\phi(Q): 0\leq Q\leq P, Q^2=Q\}=[0,\phi(P)].$$ To elaborate further, recall that in quantum mechanical terms, projections correspond to events and states correspond to probability measures. Note that if $\mu$ is the Lebesgue measure on $[0,1]$, for every Borel $E\subset [0,1]$, the set $\{\mu(F):F\subset E\}$ is the interval $[0,\mu(E)]$. This is sometimes known as the {\it divisibility property}. Our result demonstrates that on $B(\mathcal{H})$, {\it every state} has this property, and in this sense, behaves much like the Lebesgue measure.
	\medskip
	
	\section{Main results}
	
	We state and prove our first theorem.
	\medskip
	
	\begin{thm}\label{t1}Let $n\,\textgreater\, 1$ and let $\phi:M_n(\mathbb{C})\to\mathbb{C}$ be a state. Then for every $w\in\mathbb{C}$ with $|w|\leq 1$ there exists a unitary $U$ with atmost $3$ distinct eigenvalues such that $\phi(U)=w$.\end{thm}
	\begin{proof}Note that there exists a density matrix $A$ such that $$\phi(X)=\mathrm{tr}\,AX$$ for all $X\in M_n(\mathbb{C})$. Since the trace is invariant under unitary conjugation, we may assume without loss of generality that \[A=\begin{pmatrix}c_1 & &\\& \ddots &\\& & c_n \end{pmatrix}\] where each $c_j\geq 0$ and $$\sum c_j=1.$$ Then $$\phi(X)=\sum c_j\langle Xe_j, e_j\rangle$$ for all $X\in M_n(\mathbb{C})$ where $\{e_j\}_{j=1}^n$ is the standard basis of $\mathbb{C}^n$.
	\medskip
	
	We now split up into two cases. Let $n$ be even. Then $n=2k$ for some $k$. Let $S_{2,n}=\{U\in U(n):|\sigma(U)|\leq 2\}$ where $\sigma(U)$ denotes the spectrum of $U$. Observe that $S_{2,n}$ is path connected. To see this, let $U\in S_{2,n}$ and consider its spectral decomposition $$U=z_1 P + z_2 (I-P)$$ where $|z_1|=|z_2|=1$ and $P$ is a projection. Now take paths $\gamma_1,\,\gamma_2:[0,1]\to\mathbb{S}^1$ such that $\gamma_1(0)=\gamma_2(0)=1$, $\gamma_1(1)=z_1$ and $\gamma_2(1)=z_2$. Then the path $$U(t)=\gamma_1(t) P + \gamma_2(t) (I-P)$$ connects $U$ with $I$.
	Since $\phi(S_{2,n})$ is rotationally invariant, by the same argument as $\eqref{e1}$, it suffices to show that $\mathrm{ker}\,\phi\cap S_{2,n}\neq\emptyset$. Take the unitary \[\widetilde{U}=\bigoplus_{j=1}^k\begin{pmatrix} 0 & 1\\ 1 & 0\end{pmatrix}\] and note that $\phi(\widetilde{U})=0$ since all the diagonal elements of $\widetilde{U}$ are $0$. Thus, for even $n$ we can cover the closed unit disk by restricting $\phi$ to unitaries with atmost two eigenvalues. 
\medskip

Let $n$ be odd. Using unitary invariance again, write $$\phi(X)=\sum c_j\langle Xe_j, e_j\rangle$$ where the $c_j$'s are arranged in descending order, each $c_j\geq 0$ and $\sum c_j=1$. Since $$c_n=\min\, c_j,$$ $c_n\,\textless\,\frac{1}{2}$. Similar to the even case, let $S_{3,n}=\{U\in U(n): |\sigma(U)|\leq 3\}$. Then, $S_{3,n}$ is path connected and $\phi(S_{3,n})$ is rotationally invariant. Also, $I\in S_{3,n}$, so $1\in \phi(S_{3,n})$. It suffices to show the existence of $U\in S_{3,n}$ such that $\phi(U)=0$. This means we have to choose a unitary $U$ such that $$\sum_{j=1}^{n-1}\frac{c_j}{1-c_n}\langle Ue_j, e_j\rangle=\frac{-c_n\langle Ue_n, e_n\rangle}{1-c_n}.$$ Note that $$\frac{c_n}{1-c_n}\,\textless\, 1.$$ By the even case, choose a unitary $U_0\in U(n-1)$ with $|\sigma(U_0)|\leq 2$ and $$\sum_{j=1}^{n-1}\frac{c_j}{1-c_n}\langle U_0\,e_j, e_j\rangle=\frac{-c_n}{1-c_n}$$ Then $$\widetilde{U}=U_0\oplus 1$$ gives the desired unitary in $S_{3,n}$.\end{proof}
\medskip

\begin{rem2}Note that for odd $n$ we cannot necessarily cover the closed unit disk by restricting $\phi$ only to unitaries with atmost two eigenvalues. Indeed, let $$\phi(X)=\frac{\mathrm{tr}\, X}{3}$$ be the normalized trace on $M_3(\mathbb{C})$. If a unitary with zero trace has two distinct eigenvalues $z$ and $w$, either $|z|$ or $|w|$ must be $2$, which is clearly false.\end{rem2}
\medskip

\begin{cor}\label{c1}Let $\phi:M_n(\mathbb{C})\to\mathbb{C}$ be a  linear functional. Then $\{\phi(U):U\in U(n)\}=\overline{\mathbb{D}}$. \end{cor}
\begin{proof}Assume without loss of generality that $||\phi||=1$. Let $\phi(X)=\mathrm{tr}\, BX$ where $B\in M_n(\mathbb{C})$ with $\mathrm{tr}\,|B|=1$. Consider the polar decomposition $B=U|B|$ where $U$ is unitary. Then $$\phi(X)=\mathrm{tr}\, |B|XU$$ for all $X\in M_n(\mathbb{C})$. Let $|w|\leq 1$. By Theorem \ref{t1} there exists a unitary $V$ such that $\mathrm{tr}\, |B|V=0$. But this means $\phi(VU^*)=0$, completing the proof.
 \end{proof}
\medskip

Note that $U(n)$ is the set of all extreme points of the closed unit ball of $M_n(\mathbb{C})$ with the operator norm. For a general $C^*$-algebra, the closed unit ball is the closed convex hull of the unitaries, which is a famous result due to Russo and Dye \cite{rus}. 
\medskip

Corollary \ref{c1} demonstrates that for any functional $\phi$ on $M_n(\mathbb{C})$ with $||\phi||=1$, $\phi$ restricted only to the extreme points of the unit ball cover the whole disk $\overline{\mathbb{D}}$. This fails for arbitrary $C^*$-algebras, infact for $l^{\infty}_2$, as demonstrated in the introduction. Since it holds for $M_n(\mathbb{C})$ with the operator norm, it is interesting to wonder what other matrix norms have this property. 
\medskip

As an application of Theorem \ref{t1}, we show that this happens for all the Ky-Fan $k$-norms.
\medskip

\begin{lem}\label{l1}Let $\phi:M_n(\mathbb{C})\to\mathbb{C}$ be a linear functional. Then there exists a rank one operator $x\otimes y^*$ with $||x||=||y||=1$ such that $\phi(x\otimes y^*)=0$.\end{lem}
\begin{proof}As before, let $$\phi(X)=\mathrm{tr}\, BX$$ for some $B\in M_n(\mathbb{C})$. Choose unit vectors $x$ and $y$ such that $\langle Bx, y\rangle=\phi(x\otimes y^*)=0$.\end{proof}
\medskip

For $1\leq k\leq n$, the Ky-Fan $k$-norm of $X\in M_n(\mathbb{C})$ is defined by $$||X||_{(k)}=\sum_{j=1}^k\sigma_j$$ where $\sigma_1\geq\dots\geq\sigma_n$ are the singular values of $X$ arranged in descending order. This is an extremely important class of unitarily invariant norms, mainly due to the {\it Fan dominance theorem} (see \cite{rb}). 
\medskip

Note that for $k=1$ and $n$ this gives the operator norm and the trace norm respectively. We now ask, what are the extreme points of the closed unit ball $\{X\in M_n(\mathbb{C}):||X||_{(k)}\leq 1\}$? For the operator norm, the extreme points are precisely the unitaries, while for the trace norm, they are the rank one matrices $x\otimes y^*$ with $||x||=||y||=1$. 
\medskip

For $1\,\textless\, k\,\textless\, n$ the set of all extreme points is \[\mathrm{Ext}(k)=\{\frac{U}{k}:U\in U(n)\}\cup\{x\otimes y^* : ||x||=||y||=1\}.\label{e2}\tag{2}\]  This is due to Grone and Marcus (see \cite{gr, hj}). Thus, we have two connected components, unlike the case for the operator or the trace norm.
\medskip

\begin{thm}\label{t2} Let $\phi:\big(M_n(\mathbb{C}), ||.||_{(k)}\big)\to\mathbb{C}$ be a linear functional with $||\phi||=1$. Then $\phi(\mathrm{Ext(k)})=\overline{\mathbb{D}}$.\end{thm}
\begin{proof} Let $$E_1=\{\frac{U}{k}:U\in U(n)\}$$ and $$E_2=\{x\otimes y^*:||x||=||y||=1\}$$
	\medskip
	
be the two components of $\mathrm{Ext}(k)$. Both $E_1$ and $E_2$ are path connected and invariant under left and right multiplication by unitaries. By corollary \ref{c1} and lemma \ref{l1}, both $\mathrm{ker}\,\phi\,\cap\, E_1$ and $\mathrm{ker}\,\phi\,\cap\, E_2$ are non-empty. 
\medskip

Since $\phi$ attains its norm, either there exists $A\in E_1$ such that $\phi(A)=1$ or $B\in E_2$ such that $\phi(B)=1$. 
\medskip

By connectedness of $\phi(E_1)$ and $\phi(E_2)$, either $[0,1]\subset\phi(E_1)$ or $[0,1]\subset\phi(E_2)$. By rotational invariance, $\phi(\mathrm{Ext}(k))=\overline{\mathbb{D}}$.\end{proof}
\medskip

Recall that a norm $|||.|||$ on $M_n(\mathbb{C})$ is said to be unitarily invariant if $$|||UAV|||=|||A|||$$ for all $U, V\in U(n)$, $A\in M_n(\mathbb{C})$. We now ask, does Theorem \ref{t2} hold for any unitarily invariant norm on $M_n(\mathbb{C})$? It indeed does, and the proof of Theorem \ref{t2} itself suggests an approach.
\medskip

\begin{thm}\label{t3} Let $|||.|||$ be a unitarily invariant norm on $M_n(\mathbb{C})$ and let $E$ be the set of extreme points of its closed unit ball $\{X:|||X|||\leq 1\}$. Let $\phi:M_n(\mathbb{C},|||.|||)\to\mathbb{C}$ be a functional with $||\phi||=1$. Then $\phi(E)=\overline{\mathbb{D}}$.\end{thm}
\begin{proof}Since the norm is unitarily invariant, the group $U(n)$ induces an action on $E$ given by $$U\cdot X=UX$$ for all $X\in E, U\in U(n)$. This splits $E$ into a disjoint union of orbits. Given $A\in E$, let $O(A)=\{UA:U\in U(n)\}$ denote the orbit of $A$. $O(A)$ is compact, path connected and invariant under left multiplication by unitaries. 
\medskip

Consider the functional $\psi_A:M_n(\mathbb{C})\to\mathbb{C}$ given by $\psi_A(Y)=\phi(YA)$ for all $Y\in M_n(\mathbb{C})$. By Theorem \ref{t1} there exists a unitary $U$ such that $$\psi_A(U)=\phi(UA)=0.$$ This implies $\mathrm{ker}\,\phi\cap O(A)\neq\emptyset$ for all $A\in E$. Now, there exists $A_0\in E$ such that $\phi(A_0)=1$. Then $\phi(O(A_0))$ is path-connected, rotationally invariant, and contains $0$ and $1$. Therefore, $$\phi(O(A_0))=\phi(E)=\overline{\mathbb{D}}$$.
 \end{proof}
 \medskip

It is natural to ask whether Theorem \ref{t1} generalizes to infinite dimensions. 
\medskip

Let $\mathcal{H}$ be an infinite dimensional Hilbert space and let $\phi:B(\mathcal{H})\to\mathbb{C}$ be a state. Let $U(\mathcal{H})$ denote the group of unitaries on $\mathcal{H}$. Unlike the finite dimensional case, this is not compact, but still path-connected. 
\medskip

We wish to know about $\{\phi(U):U\in U(\mathcal{H})\}.$ The next two lemmas will be crucial. Given any Hilbert space $\mathcal{K}$ (can be infinite dimensional), we denote by $\mathrm{dim} \,\mathcal{K}$ the cardinality of an orthonormal basis of $\mathcal{K}$. This is consistent since any two orthonormal bases have the same cardinality. For two projections $P$ and $Q$, we say they are {\it unitarily equivalent} if there exists a unitary $U$ such that $$P=UQU^*.$$ We denote this equivalence by $P\sim Q$. We caution the reader that this notation is not standard in literature where $P\sim Q$ denotes equivalence in the Murray-von Neumannn sense (see \cite{kad, ren}).
\medskip

\begin{lem}\label{l2} Let $P$ and $Q$ be two projections in $B(\mathcal{H})$ onto subspaces $S$ and $W$ respectively. Let $S^{\perp}$ and $W^{\perp}$ be the respective orthogonal complements. Then $P\sim Q$ if and only if $\mathrm{dim}\, S=\mathrm{dim}\, W$ and $\mathrm{dim}\, S^{\perp}=\mathrm{dim}\, W^{\perp}$.\end{lem}
\begin{proof}Immediate.\end{proof}

\begin{lem}\label{l3}Let $\phi:B(\mathcal{H})\to\mathbb{C}$ be a state, where $\mathcal{H}$ is infinite dimensional. Then there exists a projection $P$ such that $P\sim(I-P)$ and $\phi(P)=\frac{1}{2}$. \end{lem}
\begin{proof}Let $\{e_{\lambda}\}_{\lambda\in J}$ be an orthonormal basis of $\mathcal{H}$, indexed by the set $J$. Since $J$ is infinite, we choose $S\subset J$ such that $|S|=|J\setminus S|=|J|$. Let $P$ be the projection onto the closed subspace $$\overline{\mathrm{span}\{e_{\lambda}:\lambda\in S\}}$$ By lemma \ref{l2}, there exists a unitary $V$ such that $$(I-P)=VPV^*.$$ If $\phi(P)=1/2$ we are done. Otherwise, assume without loss of generality that $\phi(P)\,\textless\,1/2$. Then $$\phi(I-P)=\phi(VPV^*)\,\textgreater\,\frac{1}{2}.$$ Consider the map $\xi:U(\mathcal{H})\to\mathbb{R}$ given by $$\xi(U)=\phi(UPU^*)$$ for all $U\in U(\mathcal{H})$. Since $\xi(I)\,\textless\,1/2$ and $\phi(V)\,\textgreater\,1/2$, intermediate value property gives a unitary $\widetilde{U}$ such that $$\xi(\widetilde{U})=\phi(\widetilde{U}P\widetilde{U}^*)=\frac{1}{2}.$$\end{proof}
\medskip

\begin{cor}\label{c2}Let $\phi:B(\mathcal{H})\to\mathbb{C}$ be a state. Then $\{\phi(U):U\in U(\mathcal{H})\}=\overline{\mathbb{D}}$.\end{cor}
\begin{proof}We already have rotational invariance, connectedness and $1\in\{\phi(U):U\in U(\mathcal{H})\}$. It suffices to show $$\mathrm{ker}\,\phi\cap\,U(\mathcal{H})\neq\emptyset.$$ By lemma $\ref{l3}$ there exists a projection $P$ such that $\phi(P)=1/2$. Hence, $U=2P-I$ is the desired unitary.\end{proof}
We now state and prove our next theorem.
\medskip

\begin{thm}\label{t4} Let $\mathcal{H}$ be an infinite dimensional Hilbert space and let $\phi:B(\mathcal{H})\to\mathbb{C}$ be a state. Then $\{\phi(P):P^*=P,\, P^2=P\}=[0,1]$.\end{thm}
\begin{proof}Use lemma \ref{l3} to choose a projection $P$ such that $P\sim I-P$ and $\phi(P)=1/2$. By a similar argument as in lemma \ref{l3}, choose a projection $Q\leq P$ such that $Q\sim P-Q$ and $\phi(Q)=1/4$. Let $\{x_{\alpha}\}_{\alpha\in J}$ and $\{y_{\alpha}\}_{\alpha\in J}$ be orthonormal bases of the subspaces corresponding to $Q$ and $P-Q$ respectively. Then \[\begin{aligned}\mathrm{dim}\, P&=|\{x_{\alpha}\}_{\alpha\in J}\cup\{y_{\alpha}\}_{\alpha\in J}|\\&=|J|=\mathrm{dim}\, Q=\mathrm{dim}\, P-Q.\end{aligned}\] Now, $\mathrm{dim}\, P=\mathrm{dim}\, I-P$ and $I-Q= (I-P)+(P-Q)$ is the sum of two mutually orthogonal projections of the same dimension. Hence $\dim\, Q=\dim\, I-Q$. By lemma \ref{l2} there exists a unitary $V$ such that $$I-Q=VQV^*.$$ Now $\phi(I-Q)=3/4$, so by intermediate value property, $$[\frac{1}{4},\,\frac{3}{4}]\subset\{\phi(UQU^*):U\in U(\mathcal{H})\}.$$ Repeating the argument, this time with $Q$ in place of $P$ and continuing, we see that $$[\frac{1}{2^n},\, 1-\frac{1}{2^n}]\subset\{\phi(P): P\,\,\mathrm{projection}\}$$ for all $n$. Taking union over $n$, we are done.\end{proof}
\medskip

\begin{cor}\label{c3}Let $\mathcal{H}$ be infinite dimensional and let $\phi:B(\mathcal{H})\to\mathbb{C}$ be a state. Let $P$ be a projection. Then $\{\phi(Q):0\leq Q\leq P,\, Q^2=Q\}=[0,\phi(P)]$.\end{cor}
\begin{proof}Nothing to show if $\phi(P)=0$. If not, let $\mathcal{K}=\mathrm{im}\, P$. Decompose $$\mathcal{H}=\mathcal{K}\oplus\mathcal{K}^{\perp}$$ and consider the state $\psi:B(\mathcal{K})\to\mathbb{C}$ given by $$\psi(A)=\frac{\phi(A\oplus O)}{\phi(P)}.$$ Apply Theorem \ref{t4} to complete the proof.\end{proof}
\medskip

This demonstrates the {\it divisibility property} for any state on $B(H)$. A similar result is well known for normal tracial states on a von-Neumann algebra with no minimal projection, also called a {\it diffuse} algebra. Recall that a state $\phi$ on a von-Neumann algebra $\mathcal{M}$ is said to be normal if for every bounded increasing net $\{A_{\lambda}\}$ of self-adjoint elements in $\mathcal{M}$, $$\phi(\lim\, A_{\lambda})=\sup\,\phi(A_{\lambda})$$ where $\lim\, A_{\lambda}$ denotes the strong operator limit.
\medskip

\begin{thm*}
	Let $\mathcal{M}$ be a diffuse von-Neumann
	 algebra and let $\tau:M\to\mathbb{C}$ be a normal tracial state (i.e. $\tau(AB)=\tau(BA)$ for all $A,B \in\mathcal{M}$). If $P\in\mathcal{M}$ is a projection, $\{\tau(Q):0\leq Q\leq P, Q^2=Q\}=[0,\tau(P)]$.
\end{thm*}
\medskip

See \cite{kad} for details.
\medskip

In view of Theorem \ref{t4} it is of interest to ask whether we can put any restriction on the projections so that the result still holds. This happens for normal states, but not for a general state.
\medskip

\begin{thm}\label{t5} Let $\mathcal{H}$ be an infinite dimensional Hilbert space and let $\phi:B(\mathcal{H})\to\mathbb{C}$ be a normal state. Then $\{\phi(P):P\,\,\mathrm{projection},\,\dim P\,\textless\,\infty\}=[0,1)$.\end{thm}
\begin{proof}Since $\phi$ is normal, there exists a positive, trace class operator $A$ with $\mathrm{tr}\, A=1$ (also called a density operator) such that $$\phi(X)=\mathrm{tr}\, AX$$ for all $X\in B(\mathcal{H})$ (see \cite{kad, ren}). Consider the spectral decomposition $$A=\sum\lambda_j v_j\otimes v_j^*$$ where $\{v_j\}$ is orthonormal. Choose $n$ such that $$\sum_{j=1}^n\lambda_j\,\textgreater\,1-\frac{1}{n}.$$ Let $P_n=\sum_{j=1}^nv_j\otimes v_j^*$. Then $\dim\, P=n$. Choose another rank $n$ projection $Q_n\perp P_n$. Then $$\phi(P_n)\,\textgreater\,1-\frac{1}{n}$$ and $$\phi(Q_n)\leq\frac{1}{n}$$. Since both $P_n$ and $Q_n$ have the same rank, there exists a unitary $V_n$ such that $P_n=V_nQ_nV_n^*$. By the intermediate value property, $$[\frac{1}{n},1-\frac{1}{n}]
	\subset \{\phi(P):P\,\,\mathrm{projection},\,\dim P\,\textless\,\infty\}$$ for all $n$. Taking union over $n$, we are done.  
\end{proof}

Theorem \ref{t5}, however, does not hold for non-normal states, as the following example demonstrates :
\medskip

\begin{example}Let $K(\mathcal{H})$ be the ideal of compact operators on $\mathcal{H}$. Consider the {\it Calkin algebra} $\mathcal{A}=B(\mathcal{H})/K(\mathcal{H})$ (see \cite{arv} for more details) and let $\psi:\mathcal{A}\to\mathbb{C}$ be a state. Lift $\psi$ to a state $\phi:B(\mathcal{H})\to\mathbb{C}$ given by $$\phi(X)=\psi(X+K(\mathcal{H}))$$ for all $X\in B(\mathcal{H})$. Then $\phi(X)=0$ whenever $X$ is a compact operator, in particular, $\phi(P)=0$ for all finite rank projections $P$. $\phi$ is not normal since $I$ can be approximated strongly with a net of finite rank projections.\end{example}
\medskip

\begin{rem2}Note that the range of a state $\phi$, when restricted to the operator interval $\mathcal{E}=\{O\leq A\leq I\}$ is the whole of $[0,1]$. The lattice of projections is the set of all extreme points of $\mathcal{E}$. Theorem \ref{t4} shows that restricting $\phi$ solely to the extreme points of $\mathcal{E}$ is sufficient to cover $[0,1]$.\end{rem2}
\medskip

 \begin{acknowledgement} I thank my PhD supervisor Prof. Tanvi Jain for going through the manuscript thoroughly and giving her ideas to improve the exposition. I also thank Prof. Manjunath Krishnapur for the discussion I had with him while presenting my work on this problem.   \end{acknowledgement}

\textbf{Conflict of interest :} The author declares no conflict of interest.
\medskip

\textbf{Data availability :}  No available data has been used. 
\medskip

	\bibliographystyle{amsplain}
	
 \end{document}